\documentclass[11pt, final]{article}
\usepackage{a4}
\usepackage{amsmath}%
\usepackage{amstext}%
\usepackage{amssymb}%
\usepackage{showkeys}%
\usepackage{epsfig}%
\usepackage{cite}
\usepackage{tikz}
\usepackage{subfig}
\usepackage{caption}
\usepackage{multirow}
\usepackage{booktabs}
\usepackage{floatrow}

\usepackage{listings}

\newtheorem{theorem}{Theorem}

\newtheorem{axiom}{Axiom}

\newtheorem{corollary}[theorem]{Corollary}

\newtheorem{definition}[axiom]{Definition}

\newtheorem{lemma}[theorem]{Lemma}

\newtheorem{rem}[theorem]{Remark}
\newenvironment{remark}{\rem\rm}{\endrem}

\newcounter{unnumber}

\newtheorem{prf}[unnumber]{Proof.}
\newenvironment{proof}{\prf\rm}{\hfill{$\blacksquare$}\endprf}
\newcommand{\R}{\mathbb{R}}%
\newcommand{\N}{\mathbb{N}}%
\newcommand{\ol}{\overline}%
\newcommand{\ul}{\underline}%

\renewcommand{\>}{\right\rangle}
\newcommand{\<}{\left\langle}

\DeclareMathOperator*\dom{dom}%
\DeclareMathOperator*\B{\overline{\R}}%
\DeclareMathOperator*\argmin{argmin}

\DeclareMathOperator*\crit{crit}
\DeclareMathOperator*\dist{dist}

\title{An inertial Tseng's type proximal algorithm for nonsmooth and nonconvex optimization problems}

\author{Radu Ioan Bo\c{t} \thanks{University of Vienna, Faculty of Mathematics, Oskar-Morgenstern-Platz 1, A-1090 Vienna, Austria,
email: radu.bot@univie.ac.at. Research partially supported by DFG (German Research Foundation), project BO 2516/4-1.} \and
Ern\"{o} Robert Csetnek \thanks {University of Vienna, Faculty of Mathematics, Oskar-Morgenstern-Platz 1, A-1090 Vienna, Austria,
email: ernoe.robert.csetnek@univie.ac.at. Research supported by DFG (German Research Foundation), project BO 2516/4-1.}}

\begin{document}
\maketitle

\noindent \textbf{Abstract.} We investigate the convergence of a forward-backward-forward proximal-type algorithm with inertial and memory effects when minimizing the sum of a nonsmooth function with 
a smooth one in the absence of convexity. The convergence is obtained provided an appropriate regularization of the objective satisfies the Kurdyka-\L{}ojasiewicz inequality, 
which is for instance fulfilled for semi-algebraic functions.\vspace{1ex}

\noindent \textbf{Key Words.} nonsmooth optimization, limiting subdifferential, Kurdyka-\L{}ojasiewicz inequality, Bregman distance, inertial proximal algorithm,  Tseng's type proximal algorithm\vspace{1ex}

\noindent \textbf{AMS subject classification.}  90C26, 90C30, 65K10 

\section{Introduction}\label{sec-intr}

In this work we deal with optimization problems of the form 
$$ (P) \ \inf_{x\in\R^m}[f(x)+h(x)],$$
where $f:\R^m\rightarrow (-\infty,+\infty]$ is a proper and lower semicontinuous function and $h : \R^m\rightarrow\R$ is a Fr\'{e}chet differentiable function with Lipschitz continuous gradient. 

In the full convex setting, namely when $f$ and $h$ are convex functions, a plenty of proximal-type splitting numerical schemes for solving $(P)$ is available. We mention here 
the {\it forward-backward} algorithm (see for example \cite{bauschke-book}), 
the {\it forward-backward-forward} algorithm \cite{tseng, Tse91} and also the very popular FISTA \cite{BecTeb09}, which is an accelerated version of forward-backward algorithm  under the use of step sizes considered in 
the sense of Nesterov. 

{\it Splitting} algorithms share in this contex the property that the functions $f$ and $h$ are evaluated in the iterative scheme separately. More precisely, a {\it forward step} means an evaluation of the 
smooth part through the gradient, while a {\it backward step} is nothing else than evaluating the nonsmooth counterpart via its {\it proximal operator}. The above mentioned algorithms 
have been applied when solving different real-life problems arising, for instance, in areas like image processing, multifacility location, average consensus in network coloring, support vector machines classification, clustering, etc. 
To the majority of these splitting methods inertial and memory effects have been induced, giving rise to so-called {\it inertial proximal point algorithms}. These iterative schemes have their origins in the time discretization of some differential inclusions of second order type (see \cite{alvarez2000, alvarez-attouch2001}) and share the feature that the new iterate is defined by using the previous two iterates. The increasing interest in this class of algorithms is emphasized by a considerable number of papers written in the last fifteen years on this topic,  see  
\cite{alvarez2000, alvarez-attouch2001, alvarez2004, att-peyp-red, b-c-inertial, b-c-h-inertial, mainge2008, mainge-moudafi2008, moudafi-oliny2003, cabot-frankel2011, pesq-pust}. 

The generalization of the convergence of proximal-type algorithms to the nonconvex setting is a challenging ongoing research topic. By assuming that the functions in the objective share some analytic features and 
by making consequently use of a generalization to the nonsmooth setting of  the \textit{Kurdyka-\L{}ojasiewicz property} known for smooth functions, the proximal-point algorithm for minimizing a proper and lower semicontinuous function and  
the forward-backward scheme for solving problems of the form $(P)$ have proved to possess good convergence properties also in the nonconvex case, see 
\cite{attouch-bolte2009, att-b-red-soub2010, att-b-sv2013, b-sab-teb, c-pesquet-r, f-g-peyp}. This particular class of functions, called  \textit{KL functions}, include 
semi-algebraic functions, real sub-analytic functions, semi-convex functions, uniformly convex functions, etc. (see also \cite{b-d-l-m2010, kurdyka1998, lojasiewicz1963}). 
The interest of having convergence properties in the nonconvex setting is motivated among others by applications in connection to 
sparse nonnegative matrix factorization, hard constrained feasibility, compressive sensing, etc. In what regards the latter, they give rise to the solving of optimization problems of the 
form 
$$ \min_{x\in\R^m} \left \{\lambda\|x\|_0+\frac{1}{2}\|Ax-b\|^2 \right \},$$ 
where $\lambda >0$, $\|\cdot\|_0$ is the \textit{counting norm}, $A$ is an $n\times m$ real matrix and $b\in\R^n$. Due to the fact that the counting norm is a semi-algebraic function,
algorithms for solving nonsmooth optimization problems involving KL functions represent a serious option in this sense. Let us mention that an inertial version of the forward-backward algorithm for solving 
the optimization problem $(P)$ has been proposed in \cite{ipiano}, by assuming that a regularization of the objective function 
is a KL function and that the nonsmooth function $f$ is convex.

In this paper we investigate the convergence properties of the forward-backward-forward algorithm for solving $(P)$ in the full nonconvex setting. For the backward step we use a generalization 
of the proximal operator, not only by considering it to be, as it is natural in the nonconvex setting, a set-valued mapping, but also by replacing in its standard formulation the squared-norm by the  
{\it Bregman distance} of a strongly convex and differentiable function with Lipschitz-continuous gradient. In the iterative scheme we also make use of an inertial term which assumes employing in the definition of a new iterate
the previous two iterates. The techniques for proving the convergence of the numerical scheme use the same three main ingredients, as other algorithms for nonconvex optimization problems involving 
KL functions. More precisely, we show a sufficient decrease property for the iterates, the existence of a subgradient lower bound for the iterates gap and, finally, we use some analytic features of the objective function in order 
to obtain convergence, see \cite{b-sab-teb, att-b-sv2013}. The  {\it limiting (Mordukhovich) subdifferential} and its properties play an important role in the analysis. The main result of this paper shows that, along some mild assumptions, 
provided an appropriate regularization of the objective satisfies the Kurdyka-\L{}ojasiewicz property, the convergence of the forward-backward-forward algorithm is guaranteed. As a particular instance, 
we also treat the case when the objective function is semi-algebraic and present the convergence properties of the algorithm. This makes it suitable fo solving nonsmooth optimization problems involving 
semi-algebraic functions which occurr in real-life applications, as mentioned above.

\section{Preliminaries}\label{prel}

Let us recall some notions and results which are needed in the following, see for example \cite{rock-wets}. 
Let $\N= \{0,1,2,...\}$ be the set of nonnegative integers. For $m\geq 1$, the Euclidean scalar product and the induced norm on $\R^m$
are denoted by $\langle\cdot,\cdot\rangle$ and $\|\cdot\|$, respectively. Notice that all the finite-dimensional spaces considered in the 
manuscript are endowed with the topology induced by the Euclidean norm. 

The {\it domain} of the function 
$f:\R^m\rightarrow (-\infty,+\infty]$ is defined by $\dom f=\{x\in\R^m:f(x)<+\infty\}$. We say that $f$ is {\it proper} if $\dom f\neq\emptyset$.  
Further we recall some generalized subdifferential notions and the basic properties which are needed in the paper, see \cite{boris-carte, rock-wets}. 
Let $f:\R^m\rightarrow (-\infty,+\infty]$ be a proper and lower semicontinuous function. If $x\in\dom f$, we consider the {\it Fr\'{e}chet (viscosity)  
subdifferential} of $f$ at $x$ as the set $$\hat{\partial}f(x)= \left \{v\in\R^m: \liminf_{y\rightarrow x}\frac{f(y)-f(x)-\<v,y-x\>}{\|y-x\|}\geq 0 \right \}.$$ For 
$x\notin\dom f$ we set $\hat{\partial}f(x):=\emptyset$. The {\it limiting (Mordukhovich) subdifferential} is defined at $x\in \dom f$ by 
$$\partial f(x)=\{v\in\R^m:\exists x_n\rightarrow x,f(x_n)\rightarrow f(x)\mbox{ and }\exists v_n\in\hat{\partial}f(x_n),v_n\rightarrow v \mbox{ as }n\rightarrow+\infty\},$$
while for $x \notin \dom f$, one takes $\partial f(x) :=\emptyset$.

Notice that in case $f$ is convex, these notions coincide with the {\it convex subdifferential}, which means that  
$\hat\partial f(x)=\partial f(x)=\{v\in\R^m:f(y)\geq f(x)+\<v,y-x\> \forall y\in \R^m\}$ for all $x\in\dom f$. 

Notice the inclusion $\hat\partial f(x)\subseteq\partial f(x)$ for each $x\in\R^m$. We will use the following closedness criteria 
concerning the graph of the limiting subdifferential: if $(x_n)_{n\in\N}$ and $(v_n)_{n\in\N}$ are sequences in $\R^m$ such that 
$v_n\in\partial f(x_n)$ for all $n\in\N$, $(x_n,v_n)\rightarrow (x,v)$ and $f(x_n)\rightarrow f(x)$ as $n\rightarrow+\infty$, then 
$v\in\partial f(x)$. 

The Fermat rule reads in this nonsmooth setting as: if $x\in\R^m$ is a local minimizer of $f$, then $0\in\partial f(x)$. Notice that 
in case $f$ is continuously differentiable at $x \in \R^m$ we have $\partial f(x)=\{\nabla f(x)\}$. Let us denote by 
$$\crit(f)=\{x\in\R^m: 0\in\partial f(x)\}$$ the set of {\it (limiting)-critical points} of $f$. Let us mention also the following subdifferential rule: 
if $f:\R^m\rightarrow(-\infty,+\infty]$ is proper and lower semicontinuous  and $h:\R^m\rightarrow \R$ is a continuously differentiable function, then $\partial (f+h)(x)=\partial f(x)+\nabla h(x)$ for all $x\in\R^m$. 

We turn now our attention to functions satisfying the {\it Kurdyka-\L{}ojasiewicz property}. This class of functions will play 
a crucial role in the convergence results of the proposed algorithm. For $\eta\in(0,+\infty]$, we denote by $\Theta_{\eta}$ the class of concave and continuous functions 
$\varphi:[0,\eta)\rightarrow [0,+\infty)$ such that $\varphi(0)=0$, $\varphi$ is continuously differentiable on $(0,\eta)$, continuous at $0$ and $\varphi'(s)>0$ for all 
$s\in(0, \eta)$. In the following definition (see \cite{att-b-red-soub2010, b-sab-teb}) we use also the {\it distance function} to a set, defined for $A\subseteq\R^m$ as $\dist(x,A)=\inf_{y\in A}\|x-y\|$  
for all $x\in\R^m$. 

\begin{definition}\label{KL-property} \rm({\it Kurdyka-\L{}ojasiewicz property}) Let $f:\R^m\rightarrow(-\infty,+\infty]$ be a proper and lower semicontinuous 
function. We say that $f$ satisfies the {\it Kurdyka-\L{}ojasiewicz (KL) property} at $\ol x\in \dom\partial f=\{x\in\R^m:\partial f(x)\neq\emptyset\}$ 
if there exists $\eta \in(0,+\infty]$, a neighborhood $U$ of $\ol x$ and a function $\varphi\in \Theta_{\eta}$ such that for all $x$ in the 
intersection 
$$U\cap \{x\in\R^m: f(\ol x)<f(x)<f(\ol x)+\eta\}$$ the following inequality holds 
$$\varphi'(f(x)-f(\ol x))\dist(0,\partial f(x))\geq 1.$$
If $f$ satisfies the KL property at each point in $\dom\partial f$, then $f$ is called a {\it KL function}. 
\end{definition}

The origins of this notion go back to the pioneering work of \L{}ojasiewicz \cite{lojasiewicz1963}, where it is proved that for a real-analytic function 
$f:\R^m\rightarrow\R$ and a critical point $\ol x\in\R^m$ (that is $\nabla f(\ol x)=0$), there exists $\theta\in[1/2,1)$ such that the function 
$|f-f(\ol x)|\|\nabla f\|^{-1}$ is bounded around $\ol x$. This corresponds to the situation when $\varphi(s)=s^{1-\theta}$. The result of 
\L{}ojasiewicz allows the interpretation of the KL property as a reparameterization of the function values in order to avoid flatness around the 
critical points. Kurdyka \cite{kurdyka1998} extended this property to differentiable functions definable in an o-minimal structure. 
Further extensions to the nonsmooth setting can be found in \cite{b-d-l2006, att-b-red-soub2010, b-d-l-m2010}. 

One of the remarkable properties of the KL functions is their ubiquitous in applications, according to \cite{b-sab-teb}. To the class of KL functions belong semi-algebraic, real sub-analytic, semiconvex, uniformly convex and 
convex functions satisfying a growth condition. We refer the reader to 
\cite{b-d-l2006, att-b-red-soub2010, b-d-l-m2010, b-sab-teb, att-b-sv2013, attouch-bolte2009} and the references theirin  for more details regarding all the classes mentioned above and illustrating examples. 

An important role in our convergence analysis will be played by the following uniformized KL property given in \cite[Lemma 6]{b-sab-teb}. 

\begin{lemma}\label{unif-KL-property} Let $\Omega\subseteq \R^m$ be a compact set and let $f:\R^m\rightarrow(-\infty,+\infty]$ be a proper 
and lower semicontinuous function. Assume that $f$ is constant on $\Omega$ and $f$ satisfies the KL property at each point of $\Omega$.   
Then there exist $\varepsilon,\eta >0$ and $\varphi\in \Theta_{\eta}$ such that for all $\ol x\in\Omega$ and for all $x$ in the intersection 
\begin{equation}\label{int} \{x\in\R^m: \dist(x,\Omega)<\varepsilon\}\cap \{x\in\R^m: f(\ol x)<f(x)<f(\ol x)+\eta\}\end{equation} 
the following inequality holds \begin{equation}\label{KL-ineq}\varphi'(f(x)-f(\ol x))\dist(0,\partial f(x))\geq 1.\end{equation}
\end{lemma}

We close this section by presenting two convergence results which will play a determined role in the proof of the results we provide in the next section. The first one was often used in the literature
in context of Fej\'{e}r monotonicity techniques for proving convergence results of classical algorithms for convex optimization problems or more generally for monotone inclusion problems (see \cite{bauschke-book}). 
The second one is probably also known, however we include some details of its proof for the sake of completeness. 

\begin{lemma}\label{fejer1} Let $(a_n)_{n\in\N}$ and $(b_n)_{n\in\N}$ be real sequences such that $b_n\geq 0$ for all $n\in\N$, 
$(a_n)_{n\in\N}$ is bounded below and $a_{n+1}+b_n\leq a_n$ for all $n\in\N$. Then $(a_n)_{n\in\N}$ is a monotically decreasing and convergent 
sequence and $\sum_{n\in \N}b_n< + \infty$.
\end{lemma}

\begin{lemma}\label{fejer2} Let $(\xi_n)_{n\in\N}$ and $(\varepsilon_n)_{n\in\N}$ be sequences in $[0,+\infty)$ such that $\sum_{n\in \N}\varepsilon_n<+\infty$ 
and $\xi_{n+1}\leq a\xi_n+b\xi_{n-1}+\varepsilon_n$ for all $n\geq 1$, where $a\in\R$, $b\geq 0$ and $a+b<1$. Then $\sum_{n\in \N}\xi_n<+\infty$.
\end{lemma}

\begin{proof} Fix $k\geq 1$ a positive integer. Summing up the inequality from the hypotheses for $n=1,...,k$, we obtain $\sum_{n=0}^k\xi_n+\xi_{k+1}-\xi_0-\xi_1\leq a\sum_{n=0}^k\xi_n+
b\sum_{n=0}^k\xi_n-a\xi_0-b\xi_k+\sum_{n=1}^k\varepsilon_n$. Since $\xi_n\geq 0$ for all $n\in\N$ and $b\geq 0$, we get 
$(1-a-b)\sum_{n=0}^k\xi_n\leq (1-a)\xi_0+\xi_1+\sum_{n=1}^k\varepsilon_n$ and the conclusion follows. \end{proof}

\section{An inertial forward-backward-forward algorithm}\label{sec2}

We investigate in this section the convergence properties of the inertial Tseng's type algorithm for solving nonsmooth and nonconvex optimization problems. 
We consider the following setting. \vspace{0.2cm}

\noindent {\bf Problem 1.} Let $m\geq 1$ by a positive integer, $f:\R^m\rightarrow (-\infty,+\infty]$ be a proper, lower semicontinuous function which 
is bounded from below and $h:\R^m\rightarrow\R$ a Fr\'{e}chet differentiable function such that $\nabla h$ is $L_{\nabla h}$-Lipschitz continuous with $L_{\nabla h}\geq 0$. We aim to solve the optimization problem 
\begin{equation}\label{opt-pb} (P) \ \inf_{x\in\R^m}[f(x)+h(x)] \end{equation}
by approximating the set of critical points of the objective function through a sequence generated via a forward-backward-forward algorithm of inertial-type. 

More precisely, we propose the following iterative scheme. \vspace{0.2cm}

\noindent{\bf Algorithm 1.} Chose $x_0,x_1\in\R^m$, $\ul\lambda,\ol\lambda > 0$, $\alpha\geq 0$ and the sequences 
$(\lambda_n)_{n\geq 1},(\alpha_n)_{n\geq 1}$ fulfilling $$0\leq\alpha_n\leq\alpha \ \forall n\geq 1$$ and 
$$0<\ul\lambda\leq\lambda_n\leq\ol\lambda \ \forall n\geq 1.$$ Consider the iterative scheme  
\begin{equation}\label{tseng-inertial-bregman-nonconv}(\forall n\geq 1)\hspace{0.2cm}\left\{
\begin{array}{ll}
p_n\in\argmin_{x\in\R^m}\left[f(x)+\frac{1}{\lambda_n}D_u(x,x_n)+\<x,\nabla h(x_n)\>+\frac{\alpha_n}{\lambda_n}\<x,x_{n-1}-x_n\>\right]\\
x_{n+1}=p_n+\lambda_n[\nabla h(x_n)-\nabla h(p_n)].
\end{array}\right.\end{equation}
Here,
$$D_u:\R^m\times\R^m\rightarrow\R,\ D_u(x,y)=u(x)-u(y)-\<\nabla u(y),x-y\>,$$
denotes the {\it Bregman distance} of a function $u:\R^m\rightarrow\R$ assumed to be  $\sigma$-strongly convex  with parameter $\sigma>0$ (that is 
$u-\frac{\sigma}{2}\|\cdot\|^2$ is a convex function), differentiable and such that $\nabla u$ is $L_{\nabla u}$-Lipschitz continuous with $L_{\nabla u}>0$. 

Notice that the properties of the function $u$ guarantees the following inequality (see for example \cite{bauschke-book}) 
\begin{equation}\label{ineq-D}\frac{\sigma}{2}\|x-y\|^2\leq D_u(x,y)\leq \frac{L_{\nabla u}}{2}\|x-y\|^2 \ \forall (x,y)\in\R^m\times\R^m.\end{equation}

Further, since $f$ is proper, lower semicontinuous and bounded from below and $D_u$ is coercive in its first argument 
(that is $\lim_{\|x\|\rightarrow+\infty}D_u(x,y)=+\infty$ for all $y\in\R^m$), the iterative scheme is well-defined, meaning 
that the existence of $p_n$ is guaranteed for each $n \geq 1$, since the objective function in the minimization problem to be solved at each iteration is coercive. 

Before we proceed with the convergence analysis, we discuss the relation of our scheme to other algorithms from the literature. Let us take first $u(x)=\frac{1}{2}\|x\|^2$ for all $x\in\R^m$. 
In this case $D_u(x,y)=\frac{1}{2}\|x-y\|^2$ for all $(x,y)\in\R^m\times\R^m$ and $\sigma=L_{\nabla u}=1$. The iterative scheme becomes 
\begin{equation}\label{tseng-inertial-nonconv}(\forall n\geq 1)\hspace{0.2cm}\left\{
\begin{array}{ll}
p_n\in\argmin_{x\in\R^m}\left[f(x)+\frac{1}{2\lambda_n}\left\|x-x_n+\lambda_n\nabla h(x_n)-\alpha_n(x_n-x_{n-1})\right\|^2\right]\\
x_{n+1}=p_n+\lambda_n[\nabla h(x_n)-\nabla h(p_n)]. 
\end{array}\right.\end{equation}
The convergence of this inertial Tseng's type algorithm has been analyzed in \cite{b-c-inertial} in the full convex setting, which means that $f$ and $h$ are convex
functions, in which case $p_n$ is uniquely determined and can be expressed via the {\it proximal operator} of $f$ (let us notice that
in contrast to \cite{b-c-inertial}, we do not impose here $(\alpha_n)_{n\geq 1}$ to be nondecreasing). Let us mention 
that inertial-type algorithms in the nonconvex setting have been proposed in \cite{ipiano}, where the inertial forward-backward algorithm 
from \cite{moudafi-oliny2003} has been extended from the convex setting  to KL functions, hoewever, by imposing convexity for $f$.

If we take, in addition, on the one hand, $\alpha=0$, which enforces $\alpha_n=0$ for all $n\geq 1$, then \eqref{tseng-inertial-nonconv} becomes 
\begin{equation}\label{tseng-nonconv}(\forall n\geq 1)\hspace{0.2cm}\left\{
\begin{array}{ll}
p_n\in\argmin_{x\in\R^m}\left[f(x)+\frac{1}{2\lambda_n}\left\|x-x_n+\lambda_n\nabla h(x_n)\right\|^2\right]\\
x_{n+1}=p_n+\lambda_n[\nabla h(x_n)-\nabla h(p_n)], 
\end{array}\right.\end{equation}
which is an extension to the nonconvex setting of the classical Tseng's type algorithm \cite{tseng}. The convergence of \eqref{tseng-nonconv} has been considered in \cite{tseng, br-combettes} in the full convex setting. 
Let us also mention that a forward-backward algorithm with variable metric for KL functions has been recently introduced and investigated in \cite{f-g-peyp}. 

On the other hand, if we take $h(x)=0$ for all $x\in\R^m$, the iterative scheme in \eqref{tseng-inertial-nonconv} becomes 
\begin{equation}\label{tseng-inertial-nonconv-proximal-f}(\forall n\geq 1) \ 
x_{n+1}\in\argmin_{x\in\R^m}\left[f(x)+\frac{1}{2\lambda_n}\left\|x-x_n-\alpha_n(x_n-x_{n-1})\right\|^2\right],
\end{equation} which is a proximal point algorithm with inertial and memory effects formulated in the nonconvex setting designed for finding the critical points 
of $f$. The iterative scheme without the inertial term, that is when $\alpha=0$ and, so, $\alpha_n=0$ for all $n \geq 1$, has been considered in the context of KL functions in \cite{attouch-bolte2009}. 

We proceed now with the convergence analysis of our algorithm. The following descent lemma (see for example \cite[Lemma 1.2.3]{nes}) 
will be useful in the sequel. 

\begin{lemma}\label{descent-lemma} Let $h:\R^m\rightarrow \R$ be a Fr\'{e}chet differentiable function with $L_{\nabla h}$-Lipschitz continuous gradient. Then we have 
$$h(y)\leq h(x)+\langle \nabla h(x),y-x\rangle+\frac{L_{\nabla h}}{2}\|y-x\|^2 \ \forall (x,y)\in\R^m\times\R^m.$$
\end{lemma}

\begin{lemma}\label{lemma5} In the setting of Problem 1, consider the sequences generated by Algorithm 1. Then for every  $\nu,\mu>0$ the following inequality holds 
\begin{equation}\label{f-m1-m2} (f+h)(p_n)+M_1\|x_n-p_n\|^2\leq (f+h)(p_{n-1})+M_2\|x_{n-1}-p_{n-1}\|^2 \ \forall n\geq 2,
\end{equation}
where \begin{equation}\label{m1}M_1:=\frac{\sigma}{2\ol\lambda}-L_{\nabla h}-\nu-\frac{\alpha}{\ul\lambda}\mu\end{equation}
and
\begin{equation}\label{m2}M_2:=\ol\lambda^2 L_{\nabla h}^2\left(\frac{L_{\nabla h}^2}{2\nu}+\nu+L_{\nabla h}+\frac{L_{\nabla u}}{2\ul\lambda}\right)
+\frac{\alpha}{\ul\lambda}\left(\mu\ol\lambda^2L_{\nabla h}^2+\frac{(1+\ol\lambda L_{\nabla h})^2}{2\mu}\right).\end{equation}
\end{lemma}

\begin{proof} Let us chose $\nu,\mu>0$ arbitrary and fix $n\geq 2$. The rule given in \eqref{tseng-inertial-bregman-nonconv} yields the 
inequality 
$$f(p_n)+\frac{1}{\lambda_n}D_u(p_n,x_n)+\<p_n,\nabla h(x_n)\>+\frac{\alpha_n}{\lambda_n}\<p_n,x_{n-1}-x_n\> 
$$$$\leq f(p_{n-1})+\frac{1}{\lambda_n}D_u(p_{n-1},x_n)+\<p_{n-1},\nabla h(x_n)\>+\frac{\alpha_n}{\lambda_n}\<p_{n-1},x_{n-1}-x_n\>,$$
which combined with \eqref{ineq-D} and 
$$h(p_n)\leq h(p_{n-1})+\langle \nabla h(p_{n-1}),p_n-p_{n-1}\rangle+\frac{L_{\nabla h}}{2}\|p_n-p_{n-1}\|^2$$
gives 
\begin{align}\label{ineq1} (f+h)(p_n)+\frac{\sigma}{2\lambda_n}\|p_n-x_n\|^2 \leq \ &(f+h)(p_{n-1})+\frac{L_{\nabla u}}{2\lambda_n}\|x_n-p_{n-1}\|^2+\frac{L_{\nabla h}}{2}\|p_n-p_{n-1}\|^2\nonumber\\
 &+\<\nabla h(p_{n-1})-\nabla h(x_n),p_n-p_{n-1}\> \nonumber \\
 & +\frac{\alpha_n}{\lambda_n}\<p_n-p_{n-1},x_n-x_{n-1}\>.
\end{align}
 
According to \eqref{tseng-inertial-bregman-nonconv} we have
\begin{equation}\label{x_n-p_n-1}\|x_n-p_{n-1}\| = \lambda_{n-1} \|h(x_{n-1})-h(p_{n-1})\| \leq \lambda_{n-1}L_{\nabla h}\|x_{n-1}-p_{n-1}\| 
\end{equation}
and, from here,
\begin{equation}\label{x_n-x_n-1}\|x_n-x_{n-1}\|\leq (1+\lambda_{n-1}L_{\nabla h})\|x_{n-1}-p_{n-1}\| 
\end{equation}
and
\begin{equation}\label{p_n-p_n-1}\|p_n-p_{n-1}\|^2\leq 2(\|x_n-p_n\|^2+\lambda_{n-1}^2L_{\nabla h}^2\|x_{n-1}-p_{n-1}\|^2).
\end{equation}
Moreover, we have
\begin{equation}\label{nabla h} \<\nabla h(p_{n-1})-\nabla h(x_n),p_n-p_{n-1}\>\leq \frac{\nu}{2}\|p_n-p_{n-1}\|^2+\frac{L_{\nabla h}^2}{2\nu}\|x_n-p_{n-1}\|^2
\end{equation}
and \begin{equation}\label{scal-prod} \<p_n-p_{n-1},x_n-x_{n-1}\>\leq \frac{\mu}{2}\|p_n-p_{n-1}\|^2+\frac{1}{2\mu}\|x_n-x_{n-1}\|^2.
    \end{equation}
From \eqref{ineq1}-\eqref{scal-prod} we obtain after rearranging the terms that
\begin{equation}\label{f-m1n-m2n}(f+h)(p_n)+M_{1,n}\|x_n-p_n\|^2\leq (f+h)(p_{n-1})+M_{2,n}\|x_{n-1}-p_{n-1}\|^2,\end{equation}
where $$M_{1,n}=\frac{\sigma}{2\lambda_n}-L_{\nabla h}-\nu-\frac{\alpha_n}{\lambda_n}\mu$$
and $$M_{2,n}=\lambda_{n-1}^2 L_{\nabla h}^2\left(\frac{L_{\nabla h}^2}{2\nu}+\nu+L_{\nabla h}+\frac{L_{\nabla u}}{2\lambda_n}\right)
+\frac{\alpha_n}{\lambda_n}\left(\mu\lambda_{n-1}^2L_{\nabla h}^2+\frac{(1+\lambda_{n-1} L_{\nabla h})^2}{2\mu}\right).$$
Finally, by using the bounds given for the sequences of real numbers involved, we easily derive that $M_{1,n}\geq M_1$ and $M_{2,n}\leq M_2$ and the 
conclusion follows from \eqref{f-m1n-m2n}. 
\end{proof}

\begin{lemma}\label{m1>m2} In the setting of Problem 1, consider arbitrary $\nu,\mu>0$ and chose $\ul\lambda>0$ and $\alpha\geq 0$ such that 
\begin{align}\label{l-a} 
2\ul\lambda(L_{\nabla h}+\nu)  +\ul\lambda^2L_{\nabla h}^2\left(\ul\lambda\frac{L_{\nabla h}^2}{\nu}+L_{\nabla u}+2\ul\lambda(L_{\nabla h}+\nu)\right) & \nonumber \\
 +2\alpha\left(\mu+\mu\ul\lambda^2L_{\nabla h}^2+\frac{(1+\ul\lambda L_{\nabla h})^2}{2\mu}\right) & <\sigma.
\end{align}
Then there exists $\ol\lambda>\ul\lambda$  such that the constants introduced in Lemma \ref{lemma5} fulfill $M_1>M_2$.
\end{lemma}

\begin{proof} Relation \eqref{l-a} can be equivalently written as
\begin{align*}
2\ul\lambda \left[L_{\nabla h}+\nu+\frac{\alpha}{\ul\lambda}\mu+ \ul\lambda^2 L_{\nabla h}^2\left(\frac{L_{\nabla h}^2}{2\nu} +\frac{L_{\nabla u}}{2\ul\lambda} + \nu+L_{\nabla h}\right) \right . & \\
\left . +\frac{\alpha}{\ul\lambda}\left(\mu\ul\lambda^2L_{\nabla h}^2+\frac{(1+ \ul\lambda L_{\nabla h})^2}{2\mu}\right)\right] & <\sigma.
\end{align*}
Thus there exists $\rho>0$ such that 
\begin{align}
2(\ul\lambda+\rho)\left[L_{\nabla h}+\nu+\frac{\alpha}{\ul\lambda}\mu+(\ul\lambda+\rho)^2 L_{\nabla h}^2\left(\frac{L_{\nabla h}^2}{2\nu}+\nu+L_{\nabla h}+\frac{L_{\nabla u}}{2\ul\lambda}\right)  \right . & \nonumber \\
\left . +\frac{\alpha}{\ul\lambda}\left(\mu(\ul\lambda+\rho)^2L_{\nabla h}^2+\frac{(1+(\ul\lambda+\rho) L_{\nabla h})^2}{2\mu}\right)\right] & <\sigma.
\end{align}
We define $\ol\lambda:=\ul\lambda+\rho$ and from the above inequality the relation $M_1>M_2$ follows straightforwardly. 
\end{proof}

We give now a decrease property which will be useful in the following. 

\begin{lemma}\label{x_n-p_n^2} In the setting of Problem 1, suppose that $f+h$ is bounded from below and consider the sequences generated by Algorithm 1, where 
$\nu,\mu,\ul\lambda,\ol\lambda$ and $\alpha$ are chosen as in Lemma \ref{m1>m2}. Then the following statements are true:
\begin{itemize}\item[(i)]$\sum_{n\geq 1}\|x_n-p_n\|^2<+\infty$ and 
$\sum_{n\in\N}\|x_{n+1}-x_n\|^2<+\infty$; \item[(ii)] the sequence 
$\big((f+h)(p_n)+M_2\|x_n-p_n\|^2\big)_{n\geq 1}$ is  monotically decreasing  and convergent;
 \item[(iii)] the sequence $((f+h)(p_n))_{n\geq 1}$ is convergent.\end{itemize}
\end{lemma}

\begin{proof} From Lemma \ref{lemma5} we deduce that for every $n \geq 2$
\begin{equation}\label{ineq-f}(f+h)(p_n)+M_2\|x_n-p_n\|^2+(M_1-M_2)\|x_n-p_n\|^2\leq (f+h)(p_{n-1})+M_2\|x_{n-1}-p_{n-1}\|^2.\end{equation}
The conclusion follows from Lemma \ref{m1>m2}, Lemma \ref{fejer1} and relation \eqref{x_n-x_n-1}. 
\end{proof}

The following lemma provides an estimate for some elements in the limiting subdifferential. 

\begin{lemma}\label{subdiff} In the setting of Problem 1, consider the sequences generated by Algorithm 1. Then we have for every $n \geq 2$:
\begin{equation}\label{s-n-in-subdiff}s_n\in\partial (f+h)(p_n),\end{equation}
where 
\begin{align*}
s_n = & \frac{1}{\lambda_n} \big(\nabla u(x_n)-\nabla u(p_n) \big) + \nabla h(p_n)-\nabla h(x_n)+\frac{\alpha_n}{\lambda_n}(p_{n-1}-x_{n-1})\\
& + \frac{\alpha_n\lambda_{n-1}}{\lambda_n}\big(\nabla h(x_{n-1})-\nabla h(p_{n-1})\big).
\end{align*}
Moreover, 
\begin{equation}\label{ineq-s_n}\|s_n\|\leq \left(\frac{L_{\nabla u}}{\lambda_n}+L_{\nabla h}\right)\|x_n-p_n\|+
\frac{\alpha_n}{\lambda_n}(1+\lambda_{n-1}L_{\nabla h})\|x_{n-1}-p_{n-1}\| \ \forall n\geq 2.\end{equation}
\end{lemma}

\begin{proof} Take $n\geq 2$. By using the formula for the subdifferential of the sum, from \eqref{tseng-inertial-bregman-nonconv} it follows that
$$0\in\partial f(p_n)+\frac{1}{\lambda_n} \big(\nabla u(p_n)-\nabla u(x_n) \big)+\nabla h(x_n)+\frac{\alpha_n}{\lambda_n}(x_{n-1}-x_n),$$
hence 
$$0\in\partial (f+h)(p_n)+\frac{1}{\lambda_n} \big(\nabla u(p_n)-\nabla u(x_n) \big)+\nabla h(x_n)-\nabla h(p_n)+\frac{\alpha_n}{\lambda_n}(x_{n-1}-x_n).$$
Relation \eqref{s-n-in-subdiff} follows from the above identity, by using also that
$$x_{n-1}-x_n=x_{n-1}-p_{n-1}-(x_n-p_{n-1})=x_{n-1}-p_{n-1}-\lambda_{n-1}\big(\nabla h(x_{n-1})-\nabla h(p_{n-1})\big).$$ 
The inequality \eqref{ineq-s_n} follows from the definition of the sequence $(s_n)_{n\geq 2}$. 
\end{proof}

In the following we use the notation $\omega ((p_n)_{n\geq 1})$ for the set of \textit{cluster points} of the sequence $(p_n)_{n\geq 1}$. Next we will give some properties of this set (see \cite{b-sab-teb}). 

\begin{lemma}\label{cluster-f+h} In the setting of Problem 1, suppose that the function $f+h$ is coercive (that is $\lim_{\|x\|\rightarrow +\infty}(f+h)(x)=+\infty$) and consider the sequences generated in Algorithm 1, where 
$\nu,\mu,\ul\lambda,\ol\lambda$ and $\alpha$ are chosen as in Lemma \ref{m1>m2}. Then the following statements are true: 
\begin{itemize}\item[(i)] $\emptyset\neq\omega ((p_n)_{n\geq 1})\subseteq \crit(f+h)$; 
\item[(ii)] $\lim_{n\rightarrow+\infty}\dist(p_n,\omega ((p_n)_{n\geq 1}))=0$;
\item[(iii)] $\omega ((p_n)_{n\geq 1})$ is a nonempty, compact and connected set; 
\item[(iv)] $f+h$ is finite and constant on $\omega ((p_n)_{n\geq 1})$.
\end{itemize}
\end{lemma}

\begin{proof} Since $f+h$ is a proper, lower semicontinuous and coercive function, it follows that $\inf_{x\in\R^m}[f(x)+h(x)]$ is finite and the infimum is attained 
(see \cite{rock-wets}). Hence $f+h$ is bounded from below. 

(i) According to Lemma \ref{x_n-p_n^2}(ii), we have $$(f+h)(p_n)\leq (f+h)(p_n)+M_2\|x_n-p_n\|^2\leq (f+h)(p_1)+M_2\|x_1-p_1\|^2 \ \forall n\geq 1.$$
Since the function $f+h$ is coercive, its lower level sets are bounded and we conclude that $(p_n)_{n\geq 1}$ is bounded, hence
$\omega ((p_n)_{n\geq 1})\neq\emptyset$. 

Take an arbitrary $p^*\in \omega ((p_n)_{n\geq 1})$. There exists a subsequence $(p_{n_k})_{k\in\N}$ such that $p_{n_k}\rightarrow p^*$ as 
$k\rightarrow +\infty$. We show in the following that $\lim_{k\rightarrow+\infty}f(p_{n_k})=f(p^*)$. Notice that the lower semicontinuity of the 
function $f$ ensures $\liminf_{k\rightarrow+\infty}f(p_{n_k})\geq f(p^*)$. Moreover, from \eqref{tseng-inertial-bregman-nonconv} we have that for every $n \geq 1$ 
\begin{align*}
 & f(p_n)+\frac{1}{\lambda_n}D_u(p_n,x_n)+\<p_n,\nabla h(x_n)\>+\frac{\alpha_n}{\lambda_n}\<p_n,x_{n-1}-x_n\> \\
\leq &  f(p^*)+\frac{1}{\lambda_n}D_u(p^*,x_n)+\<p^*,\nabla h(x_n)\>+\frac{\alpha_n}{\lambda_n}\<p^*,x_{n-1}-x_n\>.
\end{align*}
By using Lemma \ref{x_n-p_n^2}(i), \eqref{ineq-D} and by taking into consideration the bounds of the sequences involved, it follows
$\limsup_{k\rightarrow+\infty}f(p_{n_k})\leq f(p^*)$, hence $\lim_{k\rightarrow+\infty}f(p_{n_k})=f(p^*)$.

Further, using Lemma \ref{subdiff}, we have $s_{n_k}\in\partial (f+h)(p_{n_k})$ for all $k \geq 2$. Further, by using \eqref{ineq-s_n} and Lemma \ref{x_n-p_n^2}(i),
from $p_{n_k}\rightarrow p^*$ it follows that $s_{n_k}\rightarrow 0$ as $k\rightarrow +\infty$. Since we additionally have that $\lim_{k\rightarrow+\infty}(f+h)(p_{n_k})=(f+h)(p^*)$, 
the closedness of the graph of the limiting subdifferential operator guarantees that $0\in\partial (f+h)(p^*)$, thus $p^*\in\crit(f+h)$. 

The proof of (ii) and (iii) can be done in the lines of \cite[Lemma 5]{b-sab-teb}, by also taking into consideration \cite[Remark 5]{b-sab-teb}, where it is noticed 
that the properties (ii) and (iii) are generic for sequences satisfying $p_{n+1}-p_n\rightarrow 0$ as $n\rightarrow+\infty$.

(iv) By Lemma \ref{x_n-p_n^2}(iii), $((f+h)(p_n))_{n\geq 1}$ is a convergent sequence. Let us denote by $l\in\R$ its limit. 
Take an arbitrary $p^*\in \omega ((p_n)_{n\geq 1})$. There exists a subsequence $(p_{n_k})_{k\in\N}$ such that $p_{n_k}\rightarrow p^*$ as 
$k\rightarrow +\infty$. As shown at item (i), one has that $\lim_{k\rightarrow+\infty}(f+h)(p_{n_k})=(f+h)(p^*)$. On the other hand, 
$\lim_{k\rightarrow+\infty}(f+h)(p_{n_k})=l$, hence $(f+h)(p^*)=l$. Thus the restriction of $f+h$ to $\omega ((p_n)_{n\geq 1})$ equals $l$. 
\end{proof}

The following result characterizes the set of cluster points of the sequence $(p_n, x_n)_{n\geq 1}$.

\begin{lemma}\label{cluster-H} In the setting of Problem 1, suppose that the function $f+h$ is coercive, consider the sequences generated in Algorithm 1, where 
$\nu,\mu,\ul\lambda,\ol\lambda$ and $\alpha$ are chosen as in Lemma \ref{m1>m2}, and the constants $M_1$ and $M_2$ as in Lemma \ref{lemma5}. We introduce the function $H:\R^m\times\R^m\rightarrow \B$ defined by 
\begin{equation}\label{def-H}H(x,y)=(f+h)(x)+M_2\|x-y\|^2 \ \forall (x,y)\in\R^m\times\R^m.\end{equation}
Then the following statements are true: 
\begin{itemize}\item[(i)] $\emptyset\neq\omega ((p_n,x_n)_{n\geq 1})\subseteq \crit(H)=\{(x,x)\in\R^m\times\R^m:x\in\crit(f+h)\}$; 
\item[(ii)] $\lim_{n\rightarrow+\infty}\dist((p_n,x_n),\omega ((p_n,x_n)_{n\geq 1}))=0$;
\item[(iii)] $\omega ((p_n,x_n)_{n\geq 1})$ is a nonempty, compact and connected set; 
\item[(iv)] $H$ is finite and constant on $\omega ((p_n,x_n)_{n\geq 1})$.
\end{itemize}
\end{lemma}

\begin{proof} The proof is similar to the one of Lemma \ref{cluster-f+h} by noticing that for every $n \geq 2$ (see \eqref{ineq-f})
\begin{equation}\label{ineq-H} H(p_n,x_n)+(M_1-M_2)\|x_n-p_n\|^2\leq H(p_{n-1},x_{n-1}) 
\end{equation}
and 
\begin{equation}\label{s-n+M2}(s_n+2M_2(p_n-x_n),2M_2(x_n-p_n))\in\partial H(p_n,x_n),\end{equation}
where $(s_n)_{n \geq 2}$ is the sequence introduced in Lemma \ref{subdiff}. 
Relation \eqref{s-n+M2} follows from
$$\partial H(x,y)=\big(\partial (f+h)(x)+2M_2(x-y)\big)\times \{2M_2(y-x)\} \ \forall (x,y)\in\R^m\times\R^m.$$
\end{proof}

We are now in position to prove the convergence of the Tseng's type algorithm provided that $H$ is a KL function.

\begin{theorem}\label{th-H} In the setting of Problem 1, suppose that the function $f+h$ is coercive, consider the sequences generated in Algorithm 1, where 
$\nu,\mu,\ul\lambda,\ol\lambda$ and $\alpha$ are chosen as in Lemma \ref{m1>m2}, and the constants $M_1$ and $M_2$ as in Lemma \ref{lemma5}. We assume that
\begin{equation*}
H:\R^m\times\R^m\rightarrow \B, \ H(x,y)=(f+h)(x)+M_2\|x-y\|^2 \ \forall (x,y)\in\R^m\times\R^m,
\end{equation*}
is a KL function. Then the following statements are true:\begin{itemize}
 \item[(i)] $\sum_{n\geq 1}\|x_n-p_n\|<+\infty$ and $\sum_{n\in\N}\|x_{n+1}-x_n\|<+\infty$;
 \item[(ii)] there exists $x\in\crit(f+h)$ such that $\lim_{n\rightarrow+\infty}x_n=\lim_{n\rightarrow+\infty}p_n=x$.
\end{itemize}
\end{theorem}

\begin{proof} (i) According to Lemma \ref{cluster-H} (i) we can consider an element $p^*\in\crit(f+h)$ such that $(p^*, p^*) \in \omega ((p_n,x_n)_{n\geq 1})$. In analogy to the proof of Lemma \ref{cluster-f+h} one can
easily show that $\lim_{n\rightarrow +\infty}H(p_n,x_n)=H(p^*,p^*)$. We consider two cases. 

I. There exists $\ol n\in\N$ such that $H(p_{\ol n},x_{\ol n})=H(p^*,p^*)$. The decrease property in \eqref{ineq-H} implies 
$H(p_{n},x_{n})=H(p^*,p^*)$ for every $n\geq \ol n$. One can show inductively that the sequence $(p_n,x_n)_{n\geq \ol n}$ is constant and the 
conclusion follows. 

II. For all $n\geq 1$ we have $H(p_n,x_n)>H(p^*,p^*)$. Take $\Omega:=\omega ((p_n,x_n)_{n\geq 1})$. Since $H$ is a KL function, 
from Lemma \ref{cluster-H}(iii)-(iv) and Lemma \ref{unif-KL-property}, there exist $\varepsilon,\eta >0$ and $\varphi\in \Theta_{\eta}$ such that 
for all $(x,y)$ in the intersection 
\begin{align}\label{int-H} 
& \{(x,y)\in\R^m\times\R^m: \dist((x,y),\Omega)<\varepsilon\} \nonumber \\ 
\cap & \{(x,y)\in\R^m\times\R^m: H(p^*,p^*)<H(x,y)<H(p^*,p^*)+\eta\}\end{align} 
the following inequality holds 
\begin{equation}\label{KL-ineq-H}\varphi'(H(x,y)-H(p^*,p^*))\dist((0,0),\partial H(x,y))\geq 1.\end{equation}
Let be $n_1\geq 1$ such that $H(p_n,x_n)<H(p^*,p^*)+\eta$ for every $n\geq n_1$. 
Moreover, from Lemma \ref{cluster-H}(ii), there exists $n_2\in\N$ such that $\dist((p_n,x_n),\Omega)<\varepsilon$ for every $n\geq n_2$.
Thus the sequence $(p_n,x_n)_{n\geq N}$ belongs to the intersection in \eqref{int-H}, where $N=\max\{n_1,n_2\}$. From \eqref{KL-ineq-H}, 
we have \begin{equation}\label{KL-ineq-H-pn-xn}\varphi'(H(p_n,x_n)-H(p^*,p^*))\dist((0,0),\partial H(p_n,x_n))\geq 1 \ \forall n\geq N.\end{equation}
Further, since $\varphi$ is a concave function, we get for every $n \geq 1$ the following inequality: 
\begin{align}\label{concave} \varphi\Big(H(p_n,x_n)-H(p^*,p^*)\Big)-\varphi\Big(H(p_{n+1},x_{n+1})-H(p^*,p^*)\Big) & \geq \nonumber \\
\varphi'\Big(H(p_n,x_n)-H(p^*,p^*)\Big)\cdot\Big(H(p_n,x_n)-H(p_{n+1},x_{n+1})\Big). & \end{align}
Moreover, from \eqref{KL-ineq-H-pn-xn} and \eqref{s-n+M2} we have 
\begin{equation}\label{varphi'} \varphi'\Big(H(p_n,x_n)-H(p^*,p^*)\Big)\geq \frac{1}{\|(s_n+2M_2(p_n-x_n),2M_2(x_n-p_n))\|} \ \forall n\geq N.\end{equation}

By using for every $n \geq 1$ the notation
$$\Delta_{n,n+1}:=\varphi\big(H(p_n,x_n)-H(p^*,p^*)\big)-\varphi\big(H(p_{n+1},x_{n+1})-H(p^*,p^*)\big),$$ from 
\eqref{concave}, \eqref{varphi'} and \eqref{ineq-H} we deduce 
\begin{equation}\label{d-n-n+1} \Delta_{n,n+1}\geq (M_1-M_2)\cdot\frac{\|x_{n+1}-p_{n+1}\|^2}{\sqrt{\|s_n+2M_2(p_n-x_n)\|^2+4M_2^2\|x_n-p_n\|^2}} \ \forall n\geq N.
\end{equation}
From here  we obtain
\begin{equation}\label{arithm-mean}\|x_{n+1}-p_{n+1}\|\leq \frac{\delta}{2}\sqrt{\|s_n+2M_2(p_n-x_n)\|^2+4M_2^2\|x_n-p_n\|^2}+\frac{\Delta_{n,n+1}}{2\delta(M_1-M_2)} \ \forall n\geq N,\end{equation}
where $\delta>0$ is chosen such that the following inequality holds: 
\begin{equation}\label{delta} 
\frac{\delta\sqrt{2}}{2}\left(\sqrt{\left(\frac{L_{\nabla u}}{\ul\lambda}+L_{\nabla h}+2M_2\right)^2+4M_2^2}+\frac{\alpha}{\ul\lambda}\left(1+\ol\lambda L_{\nabla h}\right)\right)<1.
\end{equation}
Moreover, we have for every $n \geq 1$ (see \eqref{ineq-s_n})
\begin{align*}
& \sqrt{\|s_n+2M_2(p_n-x_n)\|^2+4M_2^2\|x_n-p_n\|^2}\\
\leq & \sqrt{\left[2\left(\frac{L_{\nabla u}}{\lambda_n}+L_{\nabla h}+2M_2\right)^2+4M_2^2\right]\|x_n-p_n\|^2
+2\frac{\alpha_n^2}{\lambda_n^2}\Big(1+\lambda_{n-1}L_{\nabla h}\Big)^2\|x_{n-1}-p_{n-1}\|^2}\\
\leq & \sqrt{\left[2\left(\frac{L_{\nabla u}}{\lambda_n}+L_{\nabla h}+2M_2\right)^2+4M_2^2\right]}\|x_n-p_n\|+
\sqrt{2}\frac{\alpha_n}{\lambda_n}\Big(1+\lambda_{n-1}L_{\nabla h}\Big)\|x_{n-1}-p_{n-1}\|\\
\leq & \sqrt{\left[2\left(\frac{L_{\nabla u}}{\ul\lambda}+L_{\nabla h}+2M_2\right)^2+4M_2^2\right]}\|x_n-p_n\|+
\sqrt{2}\frac{\alpha}{\ul\lambda}\Big(1+\ol\lambda L_{\nabla h}\Big)\|x_{n-1}-p_{n-1}\|.
\end{align*}

We derive from \eqref{arithm-mean} that 
\begin{equation}\label{n+1,n,n-1}\|x_{n+1}-p_{n+1}\|\leq a\|x_n-p_n\|+b\|x_{n-1}-p_{n-1}\|+\frac{\Delta_{n,n+1}}{2\delta(M_1-M_2)} \ \forall n\geq N,\end{equation}
where 
$$a:=\frac{\delta\sqrt{2}}{2}\left(\sqrt{\left(\frac{L_{\nabla u}}{\ul\lambda}+L_{\nabla h}+2M_2\right)^2+4M_2^2}\right) \ \mbox{and} \
b:=\frac{\delta\sqrt{2}}{2}\frac{\alpha}{\ul\lambda}\left(1+\ol\lambda L_{\nabla h}\right).$$ Notice that due to \eqref{delta} we have 
$a+b<1$. Now, for a fixed $k\geq 1$ we have (since $\varphi$ takes only non-negative values)
\begin{align*} \sum_{n=1}^k\Delta_{n,n+1} \ = &\ \varphi\big(H(p_1,x_1)-H(p^*,p^*)\big)-\varphi\big(H(p_{k+1},x_{k+1})-H(p^*,p^*)\big)&\\
\leq \ &\varphi\big(H(p_1,x_1)-H(p^*,p^*)\big),&\end{align*}
hence $$\sum_{n\geq 1}\frac{\Delta_{n,n+1}}{2\delta(M_1-M_2)}<+\infty.$$
From \eqref{n+1,n,n-1} and Lemma \ref{fejer2} we conclude that $\sum_{n\geq 1}\|x_n-p_n\|<+\infty$. Further, from \eqref{x_n-x_n-1} we obtain 
$\sum_{n\in\N}\|x_{n+1}-x_n\|<+\infty$. 

(ii) It follows from (i) that $(x_n)_{n\in\N}$ is a Cauchy sequence, hence it is convergent. Since $x_n-p_n\rightarrow 0$ (as $n\rightarrow+\infty$), the 
conclusion follows from Lemma \ref{cluster-H}(i). 
\end{proof}

\begin{remark}\label{ipiano} A similar condition to the one imposed in the previous theorem on the function $H$ has been used in \cite{ipiano}, for an appropriate choice of the parameter $M_2$, in order to prove the convergence 
of an inertial forward-backward algorithm for solving the problem \eqref{opt-pb} in case $f$ is a convex function.
\end{remark}

The following corollary is a direct consequence of Theorem \ref{th-H}.

\begin{corollary}\label{cor-f+h} In the setting of Problem 1, suppose that the function $f+h$ is coercive and semi-algebraic, consider the sequences generated in Algorithm 1, where 
$\nu,\mu,\ul\lambda,\ol\lambda$ and $\alpha$ are chosen as in Lemma \ref{m1>m2}. Then the following statements are true:
\begin{itemize}
 \item[(i)] $\sum_{n\geq 1}\|x_n-p_n\|<+\infty$ and $\sum_{n\in\N}\|x_{n+1}-x_n\|<+\infty$;
 \item[(ii)] there exists $x\in\crit(f+h)$ such that $\lim_{n\rightarrow+\infty}x_n=\lim_{n\rightarrow+\infty}p_n=x$.
\end{itemize} 
\end{corollary}

\begin{proof} The function $(x,y) \mapsto M_2\|x-y\|^2$ is semi-algebraic, where $M_2$ is considered as in  Lemma \ref{lemma5}. Since 
the class of semi-algebraic functions is stable under finite sums (see \cite{b-sab-teb}), it follows that $H:\R^m\times\R^m\rightarrow \B, \ H(x,y)=(f+h)(x)+M_2\|x-y\|^2$ is
semi-algebraic as well. The conclusion follows from Theorem \ref{th-H}.
\end{proof}

\end{document}